\documentclass [12pt,a4paper]{amsart}
\usepackage{amsmath}
\usepackage{cases}
\usepackage{graphicx}
\usepackage{mathrsfs}
\nonstopmode
\numberwithin{equation}{section}
\newtheorem{thm}{Theorem}[section]

\newtheorem{lem}[thm]{Lemma}

\newtheorem{rem}[thm]{Remark}

\newtheorem*{thmA}{Theorem A}
\newtheorem*{thmB}{Theorem B}

\usepackage{cases}

\newcommand{\C}{{\mathbb C}}
\newcommand{\D}{{\mathbb D}}

\newcommand{\K}{{\mathcal K}}
\newcommand{\PP}{{\mathcal P}}
\newcommand{\R}{{\mathbb R}}
\newcommand{\es}{{\mathcal S}}

\newcommand{\U}{{\mathcal D}^+}
\newcommand{\V}{{\mathcal D}^-}

\newcommand{\bD}{{\overline{\mathbb D}}}

\renewcommand{\Re}{{\,\operatorname{Re}\,}}

\newcommand{\aand}{{\quad\text{and}\quad}}

\begin{document}
\bibliographystyle{amsplain}

\title{A note on successive coefficients of convex functions}
\begin{abstract}
In this note, we investigate the supremum and the infimum of the functional
$|a_{n+1}|-|a_{n}|$ for functions, convex and analytic on the unit disk,
of the form $f(z)=z+a_2z^2+a_3z^3+\dots.$
We also consider the related problem to maximize
the functional $|a_{n+1}-a_{n}|$ for convex functions $f$
with $f''(0)=p$ for a prescribed $p\in[0,2].$
\end{abstract}
\date{\today}

\author[M. Li]{Ming Li}
\address{Graduate School of Information Sciences \\
Tohoku University \\
Aoba-ku, Sendai 980-8579, Japan}
\email{li@ims.is.tohoku.ac.jp / minglimath@sina.com}
\author[T. Sugawa]{Toshiyuki Sugawa}
\address{Graduate School of Information Sciences \\
Tohoku University \\
Aoba-ku, Sendai 980-8579, Japan}
\email{sugawa@math.is.tohoku.ac.jp}

\maketitle

\section{Introduction }
Let $\es$ be the class of normalized analytic univalent functions
$f(z)=z+\sum_{n=2}^{\infty}a_{n}z^n$ on the unit disk $\D=\{z\in\C:|z|<1\}.$
We sometimes write $a_n=a_n(f)$ to indicate the function $f.$
The Bieberbach conjecture asserts that $|a_n(f)|\le n$ for
$f\in\es$ with equality holding only for the Koebe function
$K(z)=z/(1-z)^2=z+2z^2+3z^3+\cdots$ and its rotation $e^{-i\theta}K(e^{i\theta}z).$
It had been a long-standing problem in Geometric Function Theory and
was finally proved by de Branges \cite{dB85}.
At least, soon after the conjecture was posed,
it was recognized that $|a_n(f)|\le Cn$ holds for $f\in\es$
with an absolute constant $C\le e=2.718\dots$ (see, for example,
\cite[p.~37]{Duren:univ}).
Therefore, it is somewhat surprising that the difference
$|a_{n+1}|-|a_n|$ is bounded for $f\in\es.$
Indeed, Hayman proved in his 1963 paper \cite{Hayman63} that
\begin{equation}\label{eq:diff}
||a_{n+1}|-|a_{n}||\leq A,\quad n=1, 2, 3,\cdots,
\end{equation}
for $f(z)=z+a_2z^2+\cdots$ in $\es,$
where $A\ge1$ is an absolute constant.
Note that \eqref{eq:diff} implies that $|a_n|\le An,~n=2,3,\dots.$
Unfortunately, it is known that the constant $A$ must be greater than 1.
In fact, the sharp inequalities
$$
-1\le |a_3|-|a_2|\le \frac34+e^{-\lambda_0}(2e^{-\lambda_0}-1)=1.02908\dots
$$
hold for $f(z)=z+a_2z^2+\cdots$ in $\es,$ where
$\lambda_0\approx 0.3574$ is the solution $\lambda$ in $(0,1)$ to the equation
$4\lambda e^{-\lambda}=1$ (see \cite[Theorem 3.11]{Duren:univ}).
Schaeffer and Spencer \cite{SS43} showed even that for each $n\ge2,$ there is
an {\it odd} univalent function $h\in\es$ with real coefficients such that
$|a_{2n+1}(h)|>1.$
(It is well known that $|a_3|\le 1$ for every odd univalent function
$h(z)=z+a_3z^3+a_5z^5+\cdots,$ see \cite[p.~104]{Duren:univ}.)
The problem to find the minimal value for the constant $A$ in \eqref{eq:diff}
is not solved yet.
The best result so far is the estimate $A<3.61$ due to Grinspan \cite{Grin76}.

A function $f\in S$ is called {\it starlike} (resp.~{\it convex}) if
the image $f(\D)$ is starlike with respect to the origin (resp.~convex).
The class of starlike functions is denoted by $\es^*$ and
the class of convex functions is denoted by $\K.$
In 1978 Leung \cite{Leung78} proved that \eqref{eq:diff} holds with $A=1$
for $f\in\es^*$ (see also \cite[\S 5.10]{Duren:univ}).

\begin{thmA}[Leung]
For every $f\in \es^{*}$, the following inequalities hold:
\begin{equation*}
-1\le |a_{n+1}|-|a_{n}|\leq 1,\quad n=1, 2, 3,\dots.
\end{equation*}
For each $n\ge 2,$
equality occurs in the left-hand side if and only if
$f$ is $K_\phi$ or its rotation $e^{-i\theta}K_\phi(e^{i\theta}z)$
with $\phi=k\pi/n$ for some integer $k$ with $0\le k\le n/2.$
Likewise, equality occurs in the right-hand side
if and only if $f$ is $K_\phi$ or its rotation with $\phi=k\pi/(n+1)$
for some integer $k$ with $1\le k\le(n+1)/2.$
Here,
$$
K_\phi(z)=\frac{z}{1-2z\cos\phi+z^2}
=\sum_{n=1}^\infty\frac{\sin n\phi}{\sin\phi}z^n.
$$
\end{thmA}

Note that one of $a_n$ and $a_{n+1}$ vanishes when equality holds
in the theorem.
The reader may consult \cite[\S 3.10, \S 5.9 and \S 5.10]{Duren:univ}
for more information about the difference of succesive coefficients.
In this note, we will look for a counterpart of Theorem A for convex functions.
Since the result seems to be asymmetric in this case,
to clarify the assertion, we consider the two quantities
\begin{equation}\label{eq:UV}
\U_{n}=\sup_{f\in\K}\big(|a_{n+1}(f)|-|a_{n}(f)|\big)
\aand
\V_{n}=\sup_{f\in\K}\big(|a_{n}(f)|-|a_{n+1}(f)|\big)
\end{equation}
for $n=1,2,3,\dots.$
Note that the suprema can be replaced by maxima in \eqref{eq:UV}
because of compactness of the class $\K.$
It is well known that the sharp inequalities $|a_n|\le 1$ hold
for a convex function $f(z)=z+a_2z^2+a_3z^3+\cdots,$
so that one easily gets $\U_{n}\leq 1$ and $\V_{n}\leq 1.$
When $n=1,$ we easily have $\U_1=0$ and $\V_1=1.$
From now on, we thus assume that $n\ge 2.$
We will show that $\U_{n}$ and $\V_{n}$ are much smaller than 1.
Before presenting our main results, we recall a related result due to
Robertson \cite{Rob81}.

\begin{thmB}[Robertson]
Let $f(z)=z+\sum_{n=2}^{\infty}a_{n}z^n$ be a convex function.
Then for each $n\ge2,$ the following inequality holds:
\begin{equation}
|a_{n+1}-a_{n}|\leq\frac{2n+1}{3}|a_{2}-1|.
\end{equation}
The factor $(2n+1)/3$ cannot be replaced by any smaller number independent of $f.$
\end{thmB}

The sharpness of the factor was confirmed by the fact that
$$
\frac{a_{n+1}(L_\phi)-a_n(L_\phi)}{a_2(L_\phi)-1}\to \frac{2n+1}3
$$
as $\phi\to0,$
where $L_\phi$ is the convex function given by
\begin{equation}\label{eq:L}
L_{\phi}(z)
=\frac{1}{e^{i\phi}-e^{-i\phi}}
\log\frac{1-e^{-i\phi}z}{1-e^{i\phi}z}
=\sum_{n=1}^{\infty}\frac{\sin n\phi}{n\sin\phi}z^n
\end{equation}
for $\phi\in\R.$
Here, we should take a suitable limit when $\sin\phi=0.$
For instance, $L_0(z)=\lim_{\phi\to0}L_\phi(z)=z/(1-z)=z+z^2+z^3+\cdots.$
Note the relation $K_\phi(z)=zL_\phi'(z);$ namely, $L_\phi$
is a natural counterpart of $K_\phi$ for convex functions.

\begin{thm}\label{thm:main}
Let $\U_n$ and $\V_n$ be given in \eqref{eq:UV}.
Then the following hold.
\begin{enumerate}
\item
$\U_{n}=1/(n+1)$ for $n\geq 2,$
and the function $L_{\pi/n}$ is extremal for $\U_n.$
\item
$\V_{2}=1/2,$ and $L_{\pi/3}$ is extremal for $\V_2.$
\item
$\V_{3}=1/3,$ and $L_{\pi/4}$ is extremal for $\V_3.$
\end{enumerate}
\end{thm}

In view of (ii) and (iii) in the theorem, one might expect
that $\V_n=1/n$ and that the function $L_{\pi/(n+1)}$ would be extremal
for $\V_n$ when $n\ge 4,$ as well.
It is, however, not true.
We will, in fact, prove the following.

\begin{thm}\label{thm:main2}
$1/n<\V_{n}<2/(n+1)$ for each $n\geq 4.$
\end{thm}

It is an open problem to find the value of $\V_n$ for $n\ge 4.$

As the triangle inequality implies
$||a_{n}|-|a_{n+1}||\leq |a_{n+1}-a_{n}|,$
one may think that the study of the functional $|a_{n+1}-a_{n}|$
would be helpful to our problem.
However, we immediately see that the sharp bound of $|a_{n+1}-a_n|$
for convex functions is 2 as the function $f(z)=z/(1+z)$ serves as
an extremal one.
On the other hand,
it is indeed helpful to consider the functional $|a_{n+1}-a_n|$
for refined subclasses of $\K.$
For a given number $p$ with $0\le p\le 2,$ let
$$
\K(p)=\{f\in\mathcal{K}, f''(0)=p\}.
$$
Note that the union $\K^+=\bigcup_{0\le p\le 2}\K(p)$
is smaller than the whole class $\K.$
This sort of refined subclasses of $\K$ were considered, for instance,
in \cite{Yanagihara06}.
On the other hand, each function $f$ in $\K$ is a suitable rotation of
a function in $\K(p)$ for $p=|f''(0)|.$
Since the functional $|a_{n+1}|-|a_n|$ is rotationally invariant,
one can replace $\K$ by $\K^+$ in the definition \eqref{eq:UV} of
$\U_n$ and $\V_n.$
We should note that $|a_{n+1}-a_n|$ is not necessarily invariant
under rotations.
Therefore, it might be meaningful to consider the extremal problem
maximizing $|a_{n+1}-a_n|$ among the class $\K(p)$ for a given $p\in[0,2].$
Noting the relation $p=2a_2$ for $f\in\K(p),$
Robertson's theorem (Theorem B) implies
$$
|a_3-a_2|\le \frac{5(2-p)}{6}
\aand
|a_4-a_3|\le \frac{7(2-p)}{6}
$$
for $f(z)=z+a_2z^2+a_3z^3+\cdots$ in $\K(p).$
These inequalities can be improved in the following way.

\begin{thm}\label{thm:2}
Let $0\le p\le 2.$
Suppose that $f(z)=z+a_2z^2+a_3z^3+\cdots$ is a function in $\K(p).$
Then the following sharp inequalities hold:
\begin{align}\label{eq:thm2-1}
&|a_{3}-a_{2}|\leq
\frac{(2p+1)(2-p)}6,\quad\text{and}\\
\label{eq:thm2-2}
&|a_{4}-a_{3}|\leq\left\{
\begin{aligned}
&\frac{p^3+50p^2-64p+64}{192},
~\quad&\text{when}~ 0\leq p<\frac{8}{7},\\
&\frac{-3p^3+4p^2+6p-4}{12},~\quad&\text{when}~ \frac{8}{7}\leq p\leq2.
\end{aligned}
\right.
\end{align}
Furthermore,
\begin{equation}\label{eq:thm2-3}
\sup_{f\in\K^+}|a_{3}(f)-a_{2}(f)|=\frac{25}{48}\approx 0.520833,
\end{equation}
where the supremum is attained by $L_{\phi}$
with $\phi=\arccos[3/8].$
Likewise,
\begin{equation}\label{eq:thm2-4}
\sup_{f\in\K^+}|a_{4}(f)-a_{3}(f)|=\frac{35\sqrt{70}-49}{729}\approx 0.334473,
\end{equation}
where the supremum is attained by $L_{\phi}$ with
$\phi=\arccos[(4+\sqrt{70})/18].$
\end{thm}

We will prove this theorem in Section 3.
In Section 4, we will prove Theorems \ref{thm:main} and \ref{thm:main2}.
The next section will be devoted to preparations for the proofs.

\section{Preliminaries}
Let $\PP$ denote the class of analytic functions $P$ with positive real part
on $\D$ which has the form
$$
P(z)=1+\sum_{n=1}^{\infty}p_{n}z^n.
$$
A member of $\PP$ is called a Carath\'eodory function.
Some preliminary lemmas are needed for the proof of our results.
The first one is known as Carath\'eodory's lemma
(see \cite[p.~41]{Duren:univ} for example).

\begin{lem}\label{lem:C}
For a function $P\in\PP,$
the sharp inequality $|p_{n}|\leq 2$ holds for each $n.$
\end{lem}

The sharpness can be observed through the example
$P_0(z)=(1+z)/(1-z)=1+2z+2z^2+\cdots.$
We will use also the following result due to
Carath\'eodory and Toeplitz (see \cite{GS:Toep} or \cite{Tsuji:Potential}).

\begin{lem}[Carath\'eodory-Toeplitz theorem]\label{lem:CT}
Let $P(z)=1+\sum_{n=1}^{\infty}p_{n}z^n$ be a formal power series with
complex coefficients.
Then $P$ represents a Carath\'eodory function if and only if
\begin{equation}\label{eq:Dn}
D_{n}=
\begin{vmatrix}
2 & p_{1} & p_{2} & \cdots & p_{n}\\
p_{-1} & 2 & p_{1} & \cdots& p_{n-1}\\
\vdots&\vdots&\vdots&&\vdots\\
p_{-n}&p_{-n+1}&p_{-n+2}&\cdots&2
\end{vmatrix}
\end{equation}
is non-negative for each $n\ge1,$ where $p_{-j}=\overline{p_j}$
for $j\ge1.$
Moreover, if $D_1>0,\dots, D_{k-1}>0$ and if $D_k=0,$ then
$P(z)$ is of the following form:
\begin{equation}
P(z)=\sum_{j=1}^{k}
\gamma_j
\frac{1+\varepsilon_{j}z}{1-\varepsilon_{j}z},
\quad \gamma_{j}>0,\quad|\varepsilon_{j}|=1,
\quad \varepsilon_{j}\neq\varepsilon_{h}~(j\neq h).
\end{equation}
\end{lem}

Since $P(0)=1,$ the numbers $\gamma_j$ must satisfy $\gamma_1+\dots+\gamma_k=1.$
As a special case with $k=2,$ one can deduce the following
useful assertion from Lemma \ref{lem:CT}.

\begin{lem}\label{lem:D2}
Let $P(z)=1+p_1z+p_2z^2+\cdots$ be a Carath\'eodory function
with $p_1\in\R$ and $p_2=p_1^2-2.$
Then $P$ must be of the form
$$
P(z)=\frac{1-z^2}{1-p_1z+z^2}.
$$
Moreover, the functions $f\in\es^*$ and $g\in\K$ determined by
$zf'(z)/f(z)=1+zg''(z)/g'(z)=P(z)$ have the forms
$f=K_\phi$ and $g=L_\phi,$ where $\phi=\arccos[p_1/2].$
\end{lem}

\begin{proof}
We first note that $p_1,p_2\in[-2,2]$ by Lemma \ref{lem:C}.
We next observe that $D_2=2(p_2-2)(p_2-p_1^2+2)=0$ by assumption,
where $D_2$ is given in \eqref{eq:Dn} with $n=2.$
When $D_1=0,$ which is equivalent to the condition $p_1=\pm2,$
the assertion holds clearly with $\phi=0$ or $\pi.$
Thus, we may assume that $D_1=4-p_1^2>0.$
Lemma \ref{lem:CT} now implies that $P$ has the form
$$
P(z)=\gamma_1\frac{1+\varepsilon_1 z}{1-\varepsilon_1 z}
+\gamma_2\frac{1+\varepsilon_2 z}{1-\varepsilon_2 z}
=1+2(\gamma_1\varepsilon_1+\gamma_2\varepsilon_2)z
+2(\gamma_1\varepsilon_1^2+\gamma_2\varepsilon_2^2)z^2+\cdots
$$
for $\gamma_j>0$ and $\varepsilon_j\in\partial\D~(j=1,2)$
with $\gamma_1+\gamma_2=1$ and $\varepsilon_1\ne\varepsilon_2.$
We write $\varepsilon_j=e^{i\phi_j}$ for $\phi_j\in(-\pi,\pi].$
By comparing the coefficients of $z$ and $z^2$ in the above formula,
we obtain the relations
\begin{align}
\label{eq:P1}
\gamma_1\varepsilon_1+\gamma_2\varepsilon_2&=\frac{p_1}2, \\
\label{eq:P2}
\gamma_1\varepsilon_1^2+\gamma_2\varepsilon_2^2&=\frac{p_2}2=\frac{p_1^2}2-1.
\end{align}
If one of $\varepsilon_1, \varepsilon_2$ is real, by \eqref{eq:P1},
the other must be real, too.
This occurs only when $\varepsilon_1=-\varepsilon_2=\pm1$
so that \eqref{eq:P2} implies $\gamma_1+\gamma_2=p_1^2/2-1<1,$
which contraditcs $\gamma_1+\gamma_2=1.$
Hence, $\varepsilon_1$ and $\varepsilon_2$ are both non-real; i.e.,
$\sin\phi_1\sin\phi_2\ne0.$
Taking the imaginary part of \eqref{eq:P1} and \eqref{eq:P2}, we have
$$
\begin{pmatrix}
\sin\phi_1 &\sin\phi_2 \\
\sin2\phi_1 &\sin2\phi_2
\end{pmatrix}
\begin{pmatrix}\gamma_1 \\ \gamma_2\end{pmatrix}
=\begin{pmatrix} 0\\ 0 \end{pmatrix}.
$$
Since $(\gamma_1,\gamma_2)$ is a non-zero vector,
one has
$$
\begin{vmatrix}
\sin\phi_1 &\sin\phi_2 \\
\sin2\phi_1 &\sin2\phi_2
\end{vmatrix}
=2\sin\phi_1\sin\phi_2(\cos\phi_2-\cos\phi_1)=0,
$$
which implies $\cos\phi_1=\cos\phi_2.$
Hence, we conclude that $\varepsilon_2=\bar\varepsilon_1$
and that $\gamma_1=\gamma_2=1/2.$
Put $\phi=\phi_1.$
Then, we take the real part of \eqref{eq:P1} to obtain
$\cos\phi=p_1/2.$
Thus we have seen that $P$ has the required form.
The remaining part can be easily shown by solving the differential
equations $zf'(z)/f(z)=1+zg''(z)/g'(z)=P(z).$
The proof is now complete.
\end{proof}

\begin{rem}
Under the conditions $p_1,p_2\in\R,$
as was seen in the proof, $D_2=0$  if and only if
either $p_2=p_1^2-2$ or $p_2=2.$
The latter case occurs precisely when
$$
P(z)=(1-\gamma)\frac{1+z}{1-z}+\gamma\frac{1-z}{1+z}
=1+2(1-2\gamma)z+2z^2+\cdots
$$
for some constant $\gamma\in[0,1].$
\end{rem}

By making use of the Carath\'eodory-Toeplitz theorem,
Libera and Z\l otkiewicz \cite{LZ82}
showed the following lemma.

\begin{lem}\label{lem:LZ}
Let $-2\le p\le 2$ and $p_2,p_3\in\C.$
There exists a function $P\in\PP$ with $P(z)=1+pz+p_2z^2+p_3z^3+\cdots$
if and only if
\begin{equation}\label{eq:p2}
2p_{2}=p^2+x(4-p^2).
\end{equation}
and
\begin{equation}\label{eq:p3}
4p_{3}=p^3+2(4-p^2)px-p(4-p^2)x^2+2(4-p^2)(1-|x|^2)y
\end{equation}
for some $x,y\in\C$ with  $|x|\leq 1$ and $|y|\leq 1.$
\end{lem}

\begin{rem}\label{rem:-1}
It is worth observing the following simple fact.
When \eqref{eq:p2} holds with $x=-1,$
one has the relation $p_2=p^2-2$ so that the assumption in
Lemma \ref{lem:D2} is satisfied.
Therefore, the form of $P$ can be described by Lemma \ref{lem:D2} in this case.
\end{rem}

For given real numbers $a,b,c,$ the quantity
\begin{equation}\label{eq:Y}
Y(a,b,c)=\displaystyle\max_{z\in\bD}\big(|a+bz+cz^2|+1-|z|^2\big)
\end{equation}
was used in \cite{OS15a}.
It is also helpful in the current study.

\begin{lem}\label{lem:Y}
Let $a,b,c\in\R$ with $a\ge0$ and $c\ge0.$
Then
$$
Y(a,b,c)=\begin{cases}
a+|b|+c &\quad \text{if}~ |b|\ge 2(1-c),\\
1+a+\dfrac{b^2}{4(1-c)} &\quad\text{if}~ |b|\le 2(1-c).
\end{cases}
$$
The maximum in the definition of $Y(a,b,c)$
is attained at $z=\pm 1$ in the first case according as $b=\pm|b|.$
\end{lem}

When $|b|=2(1-c)=0,$ we set $b^2/4(1-c)$ to be 0 in the above.
The lemma can easily be verified by the fact that
$$
Y(a,b,c)=\max_{0\le r\le 1}\big(a+|b|r+cr^2+1-r^2\big)
$$
in this case.
See the proof of Proposition 6 in \cite{OS15a} for details.

The next elementary result is helpful to show Lemma \ref{lem:F} below.
We recall that the discriminant $\Delta_P$ of the real
quadratic polynomial
$P(t)=\alpha+2\beta t+\gamma t^2$ is defined to be $\beta^2-\alpha\gamma.$
We will allow a degenerated case such as $\gamma=0$ to define it.
Note that $P(t)=\gamma(t+\beta/\gamma)^2-\Delta_P/\gamma\ge0$
for all $t\in\R$ if $\gamma>0$ and if $\Delta_P\le0.$

\begin{lem}\label{lem:diff}
Let $P(t)$ and $Q(t)$ be (possibly degenerated) real quadratic polynomials.
Suppose that $P>0$ and $Q>0$ on an interval $I\subset\R$
and that $\Delta_P>0.$
If there is a positive constant $T$ such that
\begin{enumerate}
\item
$\Delta_Q\ge T^{3/2}\Delta_P,$ and
\item
$TP(t)\ge Q(t)$ for $t\in I,$
\end{enumerate}
then the function $G(t)=\sqrt{P(t)}-\sqrt{Q(t)}$ is convex on $I.$
\end{lem}

\begin{proof}
By (i), we have
$$
G''(t)=-\frac{\Delta_P}{P(t)^{3/2}}+\frac{\Delta_Q}{Q(t)^{3/2}}
\ge\Delta_P\left[\left(\frac{T}{Q(t)}\right)^{3/2}-\frac1{P(t)^{3/2}}
\right].
$$
Thus the condition (ii) implies that $G''(t)\ge0$ for $t\in I.$
\end{proof}

The following technical result will be used in the proof of Theorem
\ref{thm:main}.

\begin{lem}\label{lem:F}
Let $u=6p^2/(4-p^2),~ v=2,~ a=3p^3/(4-p^2),~b=5p/2$ and $c=p/2$
for $4/3\le p\le \sqrt2$ and consider the function
$$
F(z)=|u+vz|-\big|a+bz-cz^2\big|.
$$
Then $F(z)\le F(-|z|)$ for $z\in\bD.$
\end{lem}

\begin{proof}
Fix an $r\in(0,1].$
A standard computation yields the expression $F(re^{i\theta})=G(t),$
where $t=\cos\theta,$
$$
G(t)=\sqrt{A+2Bt}-\sqrt{L+2Mt-Nt^2},
$$
and $A=u^2+v^2r^2,~ B=uvr,~ L=(a+cr^2)^2+b^2r^2,~
M=br(a-cr^2),~N=4acr^2.$
To apply Lemma \ref{lem:diff}, we put $\Delta_1=B^2,~\Delta_2=M^2+LN$
and $T=5p^2/8$ and we will show the two inequalities
\begin{equation}\label{eq:Q1}
\Delta_2\ge T^{3/2}\Delta_1
\end{equation}
and
\begin{equation}\label{eq:Q2}
T(A+2Bt)\ge L+2Mt-Nt^2\quad \text{for}~ -1\le t\le 1.
\end{equation}
We first note that
$$
\frac{\Delta_2}{\Delta_1}=\frac{(100-p^2)[6p^2+(4-p^2)r^2]^2}{48^2(4-p^2)}
\ge \frac{(100-p^2)p^4}{64(4-p^2)}.
$$
Since the last quantity is increasing in $0<p<2(\sqrt6-1),$
we have $\Delta_2/\Delta_1\ge 884/405>2$ for $4/3\le p\le\sqrt2.$
On the other hand, $T=5p^2/8\le 5/4$ and thus $T^{3/2}\le 5\sqrt5/8<2.$
The proof of \eqref{eq:Q1} is now completed.
Next we consider the quadratic polynomial
\begin{align*}
R(t)&=T(A+2Bt)-(L+2Mt-Nt^2) \\
&=
\frac{p^2[54p^4-3(4-p^2)(20-p^2)r^2-(4-p^2)^2r^4]}{4(4-p^2)^2}
+\frac{5p^2r^3}2 t
+\frac{6p^4r^2}{4-p^2}t^2.
\end{align*}
Then its discriminant is
$$
\Delta_R=-\frac{p^4r^2}{16(4-p^2)^3}H(p^2,r^2),
$$
where
$$
H(x,y)=36^2x^3-72x(4-x)(20-x)y-(100-x)(4-x)^2y^2.
$$
Since the partial derivative $H_y(x,y)$ is negative
for $0<x\le 2$ and $y>0,$ one gets
$H(p^2,r^2)\ge H(p^2,1)$ for $4/3\le p\le\sqrt2$ and $0\le r\le 1.$
It is easy to verify that $H(x,1)=1225x^3+1620x^2-4944x-1600$
is increasing in $(4/3)^2\le x.$
Hence, $H(p^2,1)\ge H((4/3)^2,1)=2^6\cdot 18379/3^6>0$
for $4/3\le p\le\sqrt2.$
Therefore, we have proved the inequality $\Delta_R<0$
and therefore \eqref{eq:Q2}.

Convexity of $G$ implies that $G(t)\le\max\{G(1),G(-1)\}
=\max\{F(r), F(-r)\}$ for $t\in[-1,1].$
We will show that $F(r)\le F(-r).$
To this end, we first observe that the polynomial
$Q(x)=a+bx-cx^2$ has two roots $x_-$ and $x_+$ with $0<-x_-<x_+$
because $c>0, b>0.$
Moreover, the inequality $Q(1)=a+b-c=a+2p>0$ implies $x_+>1.$
We note also that for $x\in\R,$  $x_-\le x\le x_+$ if and only
if $Q(x)\ge 0.$
In view of $u>v=2,$ we now obtain
\begin{align*}
F(-r)-F(r)&=[(u-vr)-(u+vr)]-\big(|Q(-r)|-|Q(r)|\big) \\
&=-4r-|Q(-r)|+|Q(r)|  \\
&=\begin{cases}
-4r+2br &\text{if}~ r<|x_-|, \\
-4r+2a-2cr^2 &\text{if}~|x_-|\le r.
\end{cases}
\end{align*}
Here, we see that $-4r+2br=(5p-4)r\ge0.$
We also have $-4r+2a-2cr^2=[(6+r^2)p^3+4p^2r-4pr^2-16r]/(4-p^2).$
By monotonicity in $p\ge 4/3,$ we observe that
$(6+r^2)p^3+4p^2r-4pr^2-16r\ge 16(24-15r-5r^2)/27>0.$
Therefore, we conclude now that $F(-r)-F(r)\ge0$ at any event,
as required.
\end{proof}

\section{Proof of Theorem \ref{thm:2}}
We recall the well-known fact that for an analytic function $f$
on $\D$ with $f(0)=0, f'(0)=1,$
$f\in\K$ if and only if $\Re[1+zf''(z)/f'(z)]>0$ on $|z|<1.$
Therefore, for $f\in\K$, there is a function $P\in\PP$ such that
\begin{equation}\label{eq:conv}
f'(z)+zf''(z)=P(z)f'(z).
\end{equation}
We write
$$
f(z)=z+\sum_{n=2}^{\infty}a_{n}z^{n}
\aand
P(z)=1+\sum_{n=1}^{\infty}p_{n}z^n.
$$
Equating the coefficients of $z^n$ in both sides of \eqref{eq:conv} for
$n=1, 2, 3,$ we obtain
\begin{equation}\label{eq:234}
a_{2}=\frac{p_{1}}{2}, \quad
a_{3}=\frac{p_{1}^2+p_{2}}{6},
\aand
a_{4}=\frac{p_1^3+3p_1p_2+2p_3}{24}.
\end{equation}
For $f\in\K(p),$ we have $p_1=2a_{2}=p.$
Therefore, by Lemma \ref{lem:LZ}, for some $x\in\bD$ we have
\begin{align*}
|a_{3}-a_{2}|&=\left|\frac{p_{1}^2+p_{2}}{6}-\frac{p_{1}}{2}\right|
=\left|\frac{x(4-p^2)}{12}-\frac{p(2-p)}{4}\right|\\
&\leq\frac{4-p^2}{12}+\frac{p(2-p)}{4}=\frac{-2p^2+3p+2}{6}
=\frac{(2p+1)(2-p)}{6}.
\end{align*}
Here, equality occurs if $x=-1.$
Thus \eqref{eq:thm2-1} has been shown.
Since $p\in[0,2]$, we have
\begin{align*}
|a_{3}-a_{2}|\le
\frac{-2p^2+3p+2}{6}
=\frac{-2(p-3/4)^2+25/8}{6}\leq\frac{25}{48},
\end{align*}
which proves \eqref{eq:thm2-3}.
Here, equalities hold simultaneously precisely when $x=-1$ and $p=3/4.$
Now Lemma \ref{lem:D2} together with Remark \ref{rem:-1}
gives the required form of the extremal function.

Next we show \eqref{eq:thm2-2}.
By substituting \eqref{eq:p2} and \eqref{eq:p3} into \eqref{eq:234}, we have
\begin{align*}
|a_{4}-a_{3}|
&=\frac1{24}\big|2p_3+3p_1p_2-4p_2+p_1^3-4p_1^2\big| \\
&=\frac1{24}\left|(4-p^2)(1-|x|^2)y+3p^3-6p^2
+\frac{(4-p^2)(5p-xp-4)x}{2}\right|\\
&\leq\frac{(4-p^2)(1-|x|^2)}{24}+\frac1{24}\left|6p^2-3p^3
+\frac{(4-p^2)((4-5p)x+x^2p)}{2}\right|\\
&\leq\frac{4-p^2}{24}Y(a,b,c),
\end{align*}
where $Y(a,b,c)$ is given in \eqref{eq:Y} and
\begin{align*}
a=\frac{3p^2}{2+p},~\quad~
b=\frac{4-5p}{2},~\quad~
c=\frac{p}{2}.
\end{align*}
Note here that equalities hold simultaneously in the above
for a suitable choice of $x$ and $y.$
A simple computation tells us that for $p\in[0,2],$
$|b|<2(1-c)$ if and only if $0<p<8/7.$
Thus Lemma \ref{lem:Y} yields
\begin{equation*}
Y(a,b,c)=
\left\{
\begin{aligned}
&1+\frac{3p^2}{2+p}+\frac{(5p-4)^2}{8(2-p)}
=\frac{p^3+50p^2-64p+64}{8(4-p^2)},
&\text{if}\quad 0\le p\le \frac{8}{7},\\
&a-b+c
=\frac{2(3p^2+2p-2)}{2+p},&\text{if}\quad 8/7\leq p\leq2.
\end{aligned}
\right.
\end{equation*}
Hence, we have the estimate $|a_4-a_3|\le \psi(p)$
for $f\in\K(p),$ where
\begin{equation}\label{eq:psi}
\psi(p)=
\left\{
\begin{aligned}
&\frac{p^3+50p^2-64p+64}{192},
\quad&\text{if}~ 0\leq p<\frac{8}{7},\\
&\frac{(2-p)(3p^2+2p-2)}{12}=\frac{-3p^3+4p^2+6p-4}{12},
\quad&\text{if}~ 8/7\leq p\leq2.
\end{aligned}
\right.
\end{equation}

Since $\psi(p)$ is convex in $0\le p\le 8/7,$
we see that $\psi(p)\le\max\{\psi(0),\psi(8/7)\}
=\max\{\frac13,\frac{103}{343}\}=1/3.$
On the other hand, we see that $\psi(p)$ takes its maximum value
in $8/7\le p\le 2$ at $p=(4+\sqrt{70})/9,$ which is
$(35\sqrt{70}-49)/729\approx 0.334473>1/3.$
Therefore, the maximum of $\psi(p)$ is taken at
$p=(4+\sqrt{70})/9$ with the choice $x=-1$ so that 
the required form of the extremal function follows from
Lemma \ref{lem:D2} and Remark \ref{rem:-1}.

\section{Proofs of Theorems \ref{thm:main} and \ref{thm:main2}}
We begin with the proof of Theorem \ref{thm:main}.

\begin{proof}[Proof of Theorem \ref{thm:main}]
By Alexander's Theorem (see \cite[Theorem 2.12]{Duren:univ}),
$f(z)=z+\sum_{n=2}^{\infty}a_{n}z^n$ is convex
if and only if $zf'(z)=z+\sum_{n=2}^{\infty}na_{n}z^n$is starlike.
Thus, by Theorem A, we have
\begin{equation}\label{eq:ThmA}
-1\le (n+1)|a_{n+1}|-n|a_{n}|\leq 1
\end{equation}
for a convex function $f$ and
equalities occur only when $f$ is a rotation of $L_\phi$
given in \eqref{eq:L} for some $\phi\in\R.$
In particular,
$$
|a_{n+1}|-|a_{n}|\leq
|a_{n+1}|-\frac{n}{n+1}|a_{n}|\leq\frac{1}{n+1}.
$$
Note that the function $f=L_{\pi/n}$ satisfies $a_n=0,~ a_{n+1}=-1/(n+1).$
Thus we have $\U_n=1/(n+1),$ which proves part (i).

As we noted in Introduction, to compute $\V_n,$ we can restrict
the range of $f$ to $\K^+=\bigcup_{0\le p\le 2}\K(p).$
For $f(z)=z+a_2z^2+a_3z^3+\cdots$ in $\K(p)$ with $0\le p\le 2,$
by \eqref{eq:234} and \eqref{eq:p2}, we have
\begin{align*}
|a_{2}|-|a_{3}|&=\frac{p_{1}}{2}-\left|\frac{p_{1}^2+p_{2}}{6}\right|
=\frac{p}{2}-\frac{|3p^2+x(4-p^2)|}{12}\\
&\leq\left\{
\begin{aligned}
&\frac{p}{2}\leq\frac{1}{2},\quad &\text{for}\quad 0\le p\leq 1,\\
&\frac{p}{2}-\left(\frac{p^2}{4}-\frac{4-p^2}{12}\right)
\leq\frac{1}{2},\quad &\text{for}\quad 1\leq p\leq 2.
\end{aligned}
\right.
\end{align*}
We have thus obtained $\V_{2}=1/2$
and equality is attained when $p=1$, $x=-1,$
which corresponds to the function $L_\phi$ with $\phi=\arccos[1/2]
=\pi/3$ by Lemma \ref{lem:D2} and Remark \ref{rem:-1}.
Now the proof of part (ii) is complete.

In the proof of Theorem \ref{thm:2},
we saw that $|a_{4}-a_{3}|\leq1/3$ when $0\le p\le 8/7.$
On the other hand, because $(-3p^3+4p^2+6p-4)/12-1/3
=-(3p-4)(p^2-2)/12,$ we see that $\psi(p)\ge 1/3$ only if
$4/3\le p\le \sqrt2.$
Hence, by virtue of Theorem \ref{thm:2}, we can harmlessly assume
that $4/3\le p\le \sqrt2$ to show $|a_3|-|a_4|\le 1/3$ for $f\in\K(p).$

By \eqref{eq:234}, \eqref{eq:p2}, \eqref{eq:p3}, we have
\begin{align*}
|a_{3}|-|a_{4}|
&=\left|\frac{p_{1}^2+p_{2}}{6}\right|
-\left|\frac{p_{1}p_{2}}{8}+\frac{p_{1}^3}{24}+\frac{p_{3}}{12}\right|\\
=&\frac{(4-p^2)(1-|x|^2)}{24}+\frac{4-p^2}{24}\left(
\left|\frac{6p^2}{4-p^2}+2x\right|
-\left|\frac{(5x-x^2)p}{2}+\frac{3p^3}{4-p^2}\right|\right).
\end{align*}
We now apply Lemma \ref{lem:diff} to get
\begin{align*}
|a_{3}|-|a_{4}|&\leq
\frac{(4-p^2)(1-r^2)}{24}
+\frac{4-p^2}{24}\left(\left(\frac{6p^2}{4-p^2}-2r\right)
-\left|\frac{(-5r-r^2)p}{2}+\frac{3p^3}{4-p^2}\right|\right)\\
&=\frac{4-p^2}{24}\left(-r^2-2r+\frac{6p^2}{4-p^2}+1
-\frac p2\left|r^2+5r-\frac{6p^2}{4-p^2}\right|\right) \\
&=:\frac{4-p^2}{24}\Phi(r).
\end{align*}

Let $r_{0}$ be the unique positive root of the quadratic
polynomial $P(r)=r^2+5r-6p^2/(4-p^2)$ in $r.$
Since $P(1)=12(2-p^2)/(4-p^2)\ge0,$ we have $0<r_0\le 1.$
When $0\leq r\leq r_{0}$,
\begin{align*}
\Phi(r)&=\left(\frac{p}{2}-1\right)r^2+\left(\frac{5p}{2}-2\right)r
+1+\frac{3p^2}{2+p},
\end{align*}
Since the right-hand side is increasing in $0\le r\le 1,$ we
have $\Phi(r)\le\Phi(r_0)$ for $0\le r\le r_0.$
When $r_{0}\leq r\leq 1$,
\begin{align*}
\Phi(r)&=\left(-\frac{p}{2}-1\right)r^2
+\left(-\frac{5p}{2}-2\right)r+1+\frac{3p^2}{2+p}.
\end{align*}
The right-hand side is decreasing in $0\le r\le 1$
so that $\Phi(r)\le\Phi(r_0)$ for $r_0\le r\le 1.$
Hence, we obtain
\begin{align*}\label{r1}
|a_{3}|-|a_{4}|&\leq \frac{4-p^2}{24}\Phi(r_0)
=\frac{4-p^2}{24}\left(-r_0^2-2r_0+\frac{6p^2}{4-p^2}+1\right) \\
&=\frac{4-p^2}{24}(3r_0+1).
\end{align*}
The inequality $(4-p^2)(3r_0+1)/24\le 1/3$ is equivalent to
$r_0\le(4+p^2)/3(4-p^2).$
This can indeed be verified as
$$
P\left(\frac{4+p^2}{3(4-p^2)}\right)
=\frac{8(2-p^2)(16-5p^2)}{9(4-p^2)^2}\ge0.
$$
Here, equality holds when $p=\sqrt2.$
By tracing back the above proof, we see that all equalities
hold, in addition, when $x=y=-1.$
Again, by Lemma \ref{lem:D2} and Remark \ref{rem:-1},
we see that an extremal function is given as $L_\phi$
with $\phi=\arccos[\sqrt2/2]=\pi/4.$
Thus the proof of part (iii) has been complete.
\end{proof}

We finish the note by proving Theorem \ref{thm:main2}.

\begin{proof}[Proof of Theorem \ref{thm:main2}]
First, we recall the fact that $|a_n|\le 1$ for a convex function
$f(z)=z+a_2z^2+\cdots$ (see \cite[p.~45]{Duren:univ}).
Thus, the first inequality in \eqref{eq:ThmA} together with this yields
$$
|a_n|-|a_{n+1}|=\frac{n|a_n|-(n+1)|a_{n+1}|}{n+1}+\frac{|a_n|}{n+1}
\le \frac2{n+1}.
$$
Here, we note that equality never holds above in view of the equality cases
in Theorem A.
Thus the right-hand inequality in the theorem has been proved.

As we noted in Introduction, the function $L_\phi$ given in \eqref{eq:L}
belongs to the class $\K.$
Therefore, we have
\begin{equation}\label{eq:lb}
\V_n\ge \max_{\phi\in\R}\Psi_n(\phi),
\end{equation}
where
$$
\Psi_n(\phi)=\left|\frac{\sin{n\phi}}{n\sin\phi}\right|
-\left|\frac{\sin(n+1)\phi}{(n+1)\sin\phi}\right|.
$$
We certainly have $\Psi_n(\theta_n)=1/n,$ which implies
$\V_n\ge 1/n.$
Here,
$$
\theta_n=\frac{\pi}{n+1}.
$$
When $n=2$ or $3,$ $1/n$ is an extremal value for $\V_n$ as we saw before.
However, this is not the case when $n\ge 4.$
Indeed, we will show that the left derivative $\Psi_n'(\theta_n-)$ is
negative for $n\ge 4,$ which implies that $\Psi_n(\theta_n-\delta)>1/n$
for a small enough $\delta>0.$
For $0<\phi<\theta_n,$ we have the expression
$$
\Psi_n(\phi)=\frac{\sin{n\phi}}{n\sin\phi}
-\frac{\sin(n+1)\phi}{(n+1)\sin\phi}.
$$
Therefore, we compute
$$
\Phi_n'(\theta_n-)=\frac{1-\frac{n+1}n\cos\theta_n}{\sin\theta_n}
=\frac{n+1}n\cdot\frac{H(\frac1{n+1})}{\sin\theta_n},
$$
where
$$
H(x)=1-x-\cos\pi x.
$$
We note that $H(x)$ is strictly convex in $0<x<1/2$ because
$H''(x)=\pi^2\cos\pi x>0$ there.
Since $H(0)=0,~ H(1/5)=(11-5\sqrt5)/20=-0.0090\dots<0,$
the inequality $H(x)<0$ holds for $0<x\le 1/5.$
Hence, $\Phi_n'(\theta_n-)<0$ for $n\ge 4$ as required.
\end{proof}

The above proof showed that
$$
\max_{\phi\in\R}\Psi_n(\phi)>\frac1n,\quad n=4,5,6,\dots.
$$
It seems difficult to find the exact value of this maximum.
For instance, a numerical computation tells us
that the maximum of $\Psi_4(\phi)$ is attained at
$\phi_4\approx 0.19834315\pi$ and its value is approximately
$0.250049846.$
See Figure 1, which was generated by Mathematica Ver.~10.2,
for the graph of the function $\Psi_4(\phi).$
The first peak is located slightly left to the bending point
$\phi=\theta_4.$
It might be meaningful to ask whether equality holds or not
in \eqref{eq:lb} for $n\ge4.$

\medskip
\noindent
{\bf Acknowledgements.}
The present note grew out of an unpublished paper on successive coefficients
of convex functions by
Professor Derek Thomas and discussions with him.
The authors would like to express their sincere thanks to him for
kind suggestions.

\begin{figure}[htbp]
\begin{center}
\includegraphics[width=0.8\textwidth]{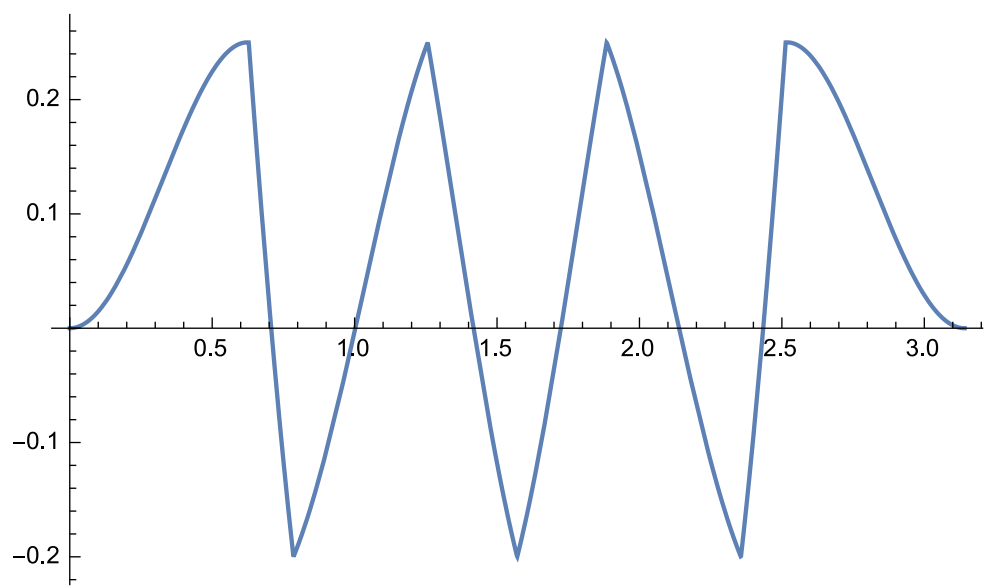}
\caption{The graph of $\Psi_4(\phi).$}
\end{center}
\end{figure}


\def\cprime{$'$} \def\cprime{$'$} \def\cprime{$'$}
\providecommand{\bysame}{\leavevmode\hbox to3em{\hrulefill}\thinspace}
\providecommand{\MR}{\relax\ifhmode\unskip\space\fi MR }
\providecommand{\MRhref}[2]{%
  \href{http://www.ams.org/mathscinet-getitem?mr=#1}{#2}
}
\providecommand{\href}[2]{#2}

\end{document}